\documentclass[12pt,a4paper]{amsart}
\usepackage{amsopn,amssymb,mathrsfs,color}
\newtheorem{thm}{Theorem}[section]
\newtheorem{cor}[thm]{Corollary}
\newtheorem{lem}[thm]{Lemma}
\newtheorem{prop}[thm]{Proposition}

\theoremstyle{definition}
\newtheorem{defn}[thm]{Definition}

\newcommand{\continuum}{\ensuremath{\mathfrak{c}}}

\newcommand{\czero}{\ensuremath{c_0}}

\newcommand{\czerok}[1]{\ensuremath{c_0({#1})}}

\newcommand{\Gdelta}{\ensuremath{G_\delta}}

\newcommand{\ind}[1]{\ensuremath{\mbox{\boldmath{$1$}}_{#1}}}

\newcommand{\mapping}[3]{\ensuremath{{#1}:{#2}\longrightarrow{#3}}}

\newcommand{\nat}{\mathbb{N}}

\newcommand{\oneton}[2]{\ensuremath{{#1}_1,\ldots,{#1}_{#2}}}

\newcommand{\pnorm}[2]{\ensuremath{\|{#1}\|_{#2}}}

\newcommand{\real}{\mathbb{R}}

\newcommand{\restrict}[1]{\ensuremath{\!\!\upharpoonright_{#1}}}
\newcommand{\setcomp}[2]{\ensuremath{\left\{{#1}\;:\;\,{#2}\right\}}}

\newcommand{\tri}{{\displaystyle |\kern-.9pt|\kern-.9pt|}}

\newcommand{\wone}{\ensuremath{\omega_1}}

\DeclareMathOperator{\card}{card} %
\DeclareMathOperator{\dom}{dom} %
\DeclareMathOperator{\ran}{ran} %

\newcommand{\st}{($*$)}

\begin{document}
\title{Tree Duplicates, $\Gdelta$-diagonals and Gruenhage spaces}
\begin{abstract}
We present an example in ZFC of a locally compact, scattered
Hausdorff non-Gruenhage space $D$ having a $\Gdelta$-diagonal. This
answers a question posed by Orihuela, Troyanski and the author in a
study of strictly convex norms on Banach spaces. In addition, we
show that the Banach space of continuous functions $C_0(D)$ admits a
$C^\infty$-smooth bump function.
\end{abstract}

\author{Richard J.\ Smith}
\address{School of Mathematical Sciences, University College Dublin, Belfield, Dublin 4, Ireland}
\email{richard.smith@ucd.ie}
\date{\today}
\maketitle

\section{Introduction}

All topological spaces considered in this note will be Hausdorff.
Recall that a norm on a Banach space is {\em strictly convex} if
every element of the unit sphere is an extreme point of the unit
ball. The authors of \cite{ost:11} introduced the following
topological property to help understand the nature of strictly
convex norms.

\begin{defn}[{\cite[Definition 2.6]{ost:11}}]\label{star}
We say that a topological space $X$ has {\st} if there exists a
sequence $(\mathscr{U}_n)_{n=1}^\infty$ of families of open subsets
of $X$, with the property that given any $x,y \in X$, there exists
$n \in \nat$ such that
\begin{enumerate}
\item $\{x,y\} \cap \bigcup\mathscr{U}_n$ is non-empty, and
\item $\{x,y\} \cap U$ is at most a singleton for all $U \in \mathscr{U}_n$.
\end{enumerate}
\end{defn}

If $(\mathscr{U}_n)_{n=1}^\infty$ satisfies properties (1) and (2)
of Definition \ref{star} then we will call it a {\em
{\st}-sequence}. This notion can be regarded as a `point-separation'
property, in the sense that it specifies in advance a family of open
sets which can separate pairs of distinct points in a controlled
way. It generalises the extensively studied {\em $\Gdelta$-diagonal}
property.

\begin{defn}
A space $X$ has a $\Gdelta$-diagonal if its diagonal is a
$\Gdelta$ set in $X^2$ or, equivalently, if there is a sequence
$(\mathscr{G}_n)_{n=1}^\infty$ of open covers of $X$, such that
given $x,y \in X$, there exists $n$ with the property that $\{x,y\}
\cap U$ is at most a singleton for all $U \in \mathscr{G}_n$.
\end{defn}

See \cite[Section 2]{gru:84} for a comprehensive introduction to spaces
with $\Gdelta$-diagonals. All spaces having a $\Gdelta$-diagonal
have {\st}, and if $L$ is a locally compact space having {\st} then
so does its 1-point compactification $L \cup \{\infty\}$: simply
adjoin to any {\st}-sequence for $L$ the singleton family $\{L\}$,
which separates all points in $L$ from $\infty$.

While compact spaces having $\Gdelta$-diagonals are metrisable (cf.\
\cite[Theorem 2.13]{gru:84}), compact spaces having {\st} can be
highly non-metrisable. The next definition presents another way in
which points can be separated by a family of open sets, over which
we have some control.

\begin{defn}[cf.\ {\cite[p.\ 372]{gru:87}}]\label{gruenhage}
A topological space $X$ is called {\em Gruenhage} if there exists a
sequence $(\mathscr{U}_n)_{n=1}^\infty$ of families of open subsets
of $X$, and sets $R_n$, $n \geq 1$, with the property that
\begin{enumerate}
\item if $x,y \in X$ then there exists $n \in \nat$ and $U \in
\mathscr{U}_n$, such that $\{x,y\} \cap U$ is a singleton, and
\item $U \cap V = R_n$ whenever $U,V \in \mathscr{U}_n$ are distinct.
\end{enumerate}
\end{defn}

Any Gruenhage space has {\st} \cite[Proposition 4.1]{ost:11}, and
there are plenty of compact non-metrisable Gruenhage spaces
\cite[Section 4]{ost:11}. As well as examples, some general
topological consequences of {\st} can be found in \cite[Section
4]{ost:11}.

The relevance of spaces having {\st} to the geometry of Banach
spaces is partly explained by the next result. Recall that a
topological space is {\em scattered} if every non-empty open subset
admits a relatively isolated point.

\begin{thm}[{\cite[Theorem 3.1]{ost:11}}]\label{scattered}
Let $K$ be a scattered compact space. Then $C(K)$ admits an
equivalent norm with a strictly convex dual norm if and only if $K$
has {\st}. Moreover, the norm can be a lattice norm.
\end{thm}

Here, $C(K)$ denotes the Banach space of continuous real-valued
functions on $K$. Gruenhage spaces were introduced in \cite{gru:87}
for reasons other than Banach space geometry, but have found a place
in this field nonetheless.

\begin{thm}[{\cite[Theorem 7]{smith:09}}]\label{smithgru}
If $K$ is Gruenhage compact then $C(K)$ admits an equivalent lattice
norm with a strictly convex dual norm.
\end{thm}

Theorem \ref{scattered} is not a consequence of Theorem
\ref{smithgru}. Under the continuum hypothesis, there is a scattered
compact, non-Gruenhage space having {\st} \cite[Example
4.10]{ost:11}. The purpose of this note is to show that such an
example exists in ZFC. For more information about how these and
related classes of topological spaces fit into Banach space theory,
we refer the reader to \cite{st:10,ost:11}.

It turns out that if $X$ has cardinality at most the continuum
$\continuum$, then we have a much more straightforward description
of Gruenhage's property available, which we will put to use in the
next section.

\begin{prop}{\cite[Proposition 2]{st:10}}\label{oddity}
Let $X$ be a topological space with $\card{X} \leq \continuum$. Then
$X$ is Gruenhage if and only if there is a sequence
$(U_n)_{n=1}^\infty$ of open subsets of $X$ with the property that
if $x,y \in X$, then $\{x,y\} \cap U_n$ is a singleton for some $n$.
\end{prop}

The basic idea behind the example remains the same as that of
\cite[Example 4.10]{ost:11}. We take a topological space
$X$ of cardinality $\continuum$ and endow a `duplicate' $D = X \times \{1,-1\}$
with a new topology, the basic open sets of which use the existing
structure of $X$ to `oscillate' rapidly between the levels $+1$ and
$-1$. This oscillation induces a non-trivial interaction between
the levels and will make it difficult to separate all the `problem
pairs' of the form $(x,1),(x,-1)$, $x \in X$, in the manner of
Proposition \ref{oddity}. This will render the space non-Gruenhage.
However, at the same time, we have construct $D$ delicately enough to
ensure that we don't lose the other properties that we want it to have.

We shall use a particular tree in its standard interval topology
as our starting point. Before proceeding with the construction, we
should point out that trees by themselves cannot furnish us with a
desired example. According to \cite[Theorem 4.6]{ost:11}, if $\Upsilon$
is any tree which is Hausdorff in its standard interval topology,
then $\Upsilon$ is Gruenhage if and only if it has {\st}. In particular,
$\Upsilon$ has a $\Gdelta$-diagonal in its standard interval
topology if and only if it is {\em $\real$-embeddable} \cite{hart:82},
which in turn means that it is certainly Gruenhage; see e.g.\ \cite[pp.\
20 -- 21]{ost:11} for more details.
%

\section{The $\Lambda$-Duplicate}

Recall that a {\em tree} is a partially ordered set
$(\Upsilon,\preccurlyeq)$ with the property that, given any $t \in
\Upsilon$, the set of predecessors $(0,t] = \setcomp{s \in
\Upsilon}{s \preccurlyeq t}$ is well ordered. For convenience, we
shall regard $0$ as an extra element, not in $\Upsilon$, such that
$0 \prec t$ for all $t \in \Upsilon$. Trees are natural
generalisations of ordinal numbers. We will use standard interval
notation throughout this note. For instance $(r,t]$, where $t \in
\Upsilon$ and $r \in \Upsilon \cup \{0\}$, is the set of all $s
\in\Upsilon$ satisfying $r \prec s \preccurlyeq t$. Other intervals
such as $[r,t)$ are defined accordingly. For further notation and
details about trees, we refer the reader to e.g.\ \cite{haydon:99}.

The tree in question was first considered by
Kurepa. Denoted by $\Lambda$ in \cite[Section 10]{haydon:99}, Kurepa's tree
is the set of injective functions $\mapping{t}{\alpha}{\omega}$ with
(countable) ordinal domain and coinfinite range, and where $s
\preccurlyeq t$ if and only if $t$ extends $s$. We shall regard
functions in the usual set-theoretic sense, that is, as sets of
ordered pairs, and with $\dom f$ and $\ran f$ the domain and range
of a function $f$, respectively. In this note, we treat $0$ as distinct
from the empty function $\varnothing$, the latter being the least
element of $\Lambda$.

A subset $A \subseteq \Lambda \cup \{0\}$ is
an {\em antichain} if no two distinct elements of $A$ are comparable
in the tree order. We shall define $\Lambda^+$ to be the set of
elements of $\Lambda$ with successor ordinal domain. It is widely
known and easy to show
that $\{0,\varnothing\} \cup \Lambda^+$ can be written as a countable
union of antichains: if
\[
A_0 = \{0\},\quad A_1 = \{\varnothing\} \quad\mbox{and}\quad A_{n+2} = \setcomp{t \in \Lambda^+}{t(\dom t-1) = n}
\]
then each $A_n$ is an antichain and $\{0,\varnothing\} \cup \Lambda^+ = \bigcup_{n=0}^\infty
A_n$.

The underlying set of our example is $D = \Lambda \times \{1,-1\}$.
We set up a function $\tau$ on pairs $(s,t)$, $s \in \Lambda \cup
\{0\}$, $t \in \Lambda$, $s \preccurlyeq t$, which will be central
to the definition of our topology on $D$. Given $s \prec t$, we
define a finite sequence of ordinals $\dom t = \beta_0
> \beta_1 > \beta_2 > \dots > \beta_k = \dom s$. Given $\beta_i
> \dom s$, let $\beta_{i+1} \in [\dom s,\beta_i)$ be the unique
ordinal with the property that
\[
t(\beta_{i+1}) \leq t(\xi) \qquad \mbox{ for all }\xi \in [\dom
s,\beta_i).
\]
As the $\beta_i$ are strictly decreasing, this process eventually
stops at some finite stage $k > 0$, with $\beta_k = \dom s$. Let
\[
\tau(s,t) = (\beta_k,\dots,\beta_1).
\]
For convenience, we also set $\tau(t,t)$ to be the empty sequence
for each $t \in \Lambda$.
%

The next lemma will help when we use the $\tau$
sequences to define a basis for our topology. We let $^\frown$ denote
concatenation of sequences.

\begin{lem}\label{localext}
Given $t,u \in \Lambda$, $t \prec u$, there exists $r \in \Lambda \cup \{0\}$, $r \prec
t$, such that $\tau(s,u)= \tau(s,t)^\frown\tau(t,u)$ for every $s
\in (r,t]$.
\end{lem}

\begin{proof}
If $t$ is the empty function $\varnothing$ then let $r=0$. If $t \in
\Lambda^+$ then we let $r = t\restrict{\dom t-1}$ be the immediate
predecessor of $t$ in the tree order. Now suppose that $\dom t$ is a
limit ordinal, and let $\tau(t,u) = (\beta_k,\dots,\beta_1)$. By
construction, we have
\[
u(\beta_1) < u(\beta_2) < \dots < u(\beta_k) = u(\dom t).
\]
Since $\dom t$ is a limit, there exists $\alpha < \dom t$ such that
$t(\eta) = u(\eta) > u(\dom t)$ whenever $\eta \in [\alpha,\dom t)$.
Set $r=t\restrict\alpha$, so that $\dom r = \alpha$. Let $s \in
(r,t]$ and
\[
\tau(s,u) = (\gamma_m,\dots,\gamma_1).
\]
By the choice of $\alpha$, we have ensured that $m \geq k$ and
$\gamma_i = \beta_i$ whenever $i \leq k$.
\end{proof}

It is time to define the basic open sets. Let $\ell(s,t)$ denote the
length of $\tau(s,t)$. Given $(t,i) \in D$ and $r \in \Lambda \cup
\{0\}$, $r \prec t$, let
\[
W(r,t,i) = \setcomp{(s,j) \in D}{s \in (r,t] \mbox{ and } j =
(-1)^{\ell(s,t)}i}.
\]
Observe that if $\mapping{\pi}{D}{\Lambda}$ is the natural
projection, then the restriction of $\pi$ to any $W(r,t,i)$ is
injective. Moreover, the images $\pi(W(r,t,i)) = (r,t]$ form the
usual basis of the standard interval topology on $\Lambda$.

\begin{prop}\label{basis}
The $W(r,t,i)$ form a basis for a locally compact scattered topology
on $D$.
\end{prop}

\begin{proof}
First, we show that these sets form a basis. If $(t,k) \in
W(r_1,u_1,i_1) \cap W(r_2,u_2,i_2)$ then, by Lemma \ref{localext}
and the fact that $r_1,r_2 \prec t$ are comparable, we can find $r
\in [\max\{r_1,r_2\},t)$ such that $\tau(s,u_j) =
\tau(s,t)^\frown\tau(t,u_j)$ whenever $s \in (r,t]$ and $j = 1,\,2$.
It follows that
\[
W(r,t,k) \subseteq W(r_1,u_1,i_1) \cap W(r_2,u_2,i_2).
\]
Indeed, if $(s,l) \in W(r,t,k)$ then $s \in (r,t] \subseteq
(r_j,u_j]$ and
\[
l = (-1)^{\ell(s,t)}k = (-1)^{\ell(s,t)}(-1)^{\ell(t,u_j)}i =
(-1)^{\ell(s,u_j)}i
\]
since $\ell(s,u_j) = \ell(s,t) + \ell(t,u_j)$. Therefore $(s,l) \in
W(r_1,u_1,i_1) \cap W(r_2,u_2,i_2)$ as required. We conclude that the
$W(r,t,i)$ form a basis for a topology on $D$.

Now we show that this topology is Hausdorff and scattered. Let
$(t_1,i_1), (t_2,i_2) \in D$ be distinct. If $t_1 \neq t_2$ then we
let $r$ be the largest common predecessor of these elements. It is
clear that $W(r,t_1,i_1) \cap W(r,t_2,i_2)$ is empty. Instead, if
$t_1 = t_2$ then $i_1 = -i_2$, so $W(0,t_1,i_1) \cap W(0,t_2,i_2)$
is empty. To see that the topology is scattered,
let $E \subseteq D$ be non-empty and find minimal $t \in \Lambda$,
subject to there being some $i$ for which $(t,i) \in E$. Then
$W(0,t,i) \cap E = \{(t,i)\}$.

Finally, we show that each $W(r,t,i)$ is compact. Suppose that
$(u,k) \in W(r,v,i) \cap U$, where $U$ is some open set. Again from
Lemma \ref{localext}, we know that we can find $s \in [r,u)$ such
that $\tau(t,v) = \tau(t,u)^\frown\tau(u,v)$ whenever $t \in (s,u]$.
Moreover, we can choose $s$ so that $W(s,u,k) \subseteq U$. If $t
\in (s,u]$ and $(t,l) \in W(r,v,i)$, then we have
\begin{eqnarray*}
l &=& (-1)^{\ell(t,v)}i \\
&=& (-1)^{\ell(t,v)}(-1)^{-\ell(u,v)}k \qquad\mbox{since }k =
(-1)^{\ell(u,v)}i\\
&=& (-1)^{\ell(t,u)}k \qquad\mbox{since }\ell(t,v) = \ell(t,u) +
\ell(u,v)
\end{eqnarray*}
and so $(t,l) \in W(s,u,k) \subseteq U$.

This will allow us to show that $W(r,v,i)$ is compact. The method
follows that used to show that each $(r,t]$ is compact in the usual
interval topology of $\Lambda$. If $W(r,v,i)$ is covered by a family
of open sets $\mathscr{U}$, we can find $U_1 \in \mathscr{U}$
covering $(v_1,i_1)$, where $v_1 = v$ and $i_1 = i$. From above,
there is some $v_2 \prec v_1$ such that $(t,l) \in U_1$ whenever
$(t,l) \in W(r,v,i)$ and $t \in (v_2,v_1]$. Then we pick $U_2 \in
\mathscr{U}$ covering $(v_2,i_2)$, where $i_2$ is the unique number
satisfying $(v_2,i_2) \in W(r,v,i)$, and continue. The process stops
at some finite $k > 1$, with $v_k = r$ and $W(r,v,i)$ covered by
$U_1,\dots,U_{k-1}$.
\end{proof}

\begin{defn}
We shall call $D$ above, together with this topology, the {\em
$\Lambda$-duplicate}.
\end{defn}

\begin{thm}
The $\Lambda$-duplicate has a $\Gdelta$-diagonal but is not
Gruenhage.
\end{thm}

\begin{proof}
First, we show that $D$ has a $\Gdelta$-diagonal. Given $s \prec t$
and $\tau(s,t) = (\beta_k,\dots,\beta_1)$, we define $p(s,t) =
t(\beta_1)$. Note that $p(s,t) \leq t(\beta_k) = t(\dom s)$. We'll
set $p(t,t) = \infty$ for every $t \in \Lambda$, again for convenience.
For $(u,i) \in D$ and finite $p$, define
\[
V(u,i,p) = \setcomp{(t,j)}{t \preccurlyeq u,\,p(t,u) \geq p \mbox{
and } j = (-1)^{\ell(t,u)}i}.
\]
If $(t,j) \in V(u,i,p)$ then, by Lemma \ref{localext}, there exists
$r \prec t$ such that whenever $s \in (r,t]$, we have $\tau(s,u) =
\tau(s,t)^\frown\tau(t,u)$. Certainly, for such $s$, we get $p(s,u)
= p(t,u) \geq p$ and
\[
W(r,t,j) \subseteq V(u,i,p).
\]
Therefore, each $V(u,i,p)$ is open. We claim that if
\[
\mathscr{G}_p = \setcomp{V(u,i,p)}{(u,i) \in D}
\]
then $(\mathscr{G}_p)_{p=1}^\infty$ forms a $\Gdelta$-diagonal
sequence for $D$. Let $(u_1,i_1), (u_2,i_2) \in D$ be distinct, and
suppose that for some $(u,i) \in D$ and $p$ we have $(u_1,i_1),
(u_2,i_2) \in V(u,i,p)$. Since $u_1,\,u_2 \preccurlyeq u$, they are
comparable. Necessarily, $u_1 \neq u_2$, for otherwise we would have
\[
i_1 = (-1)^{\ell(u_1,u)}i = (-1)^{\ell(u_2,u)}i = i_2,
\]
giving $(u_1,i_1) = (u_2,i_2)$. Without loss of generality, assume
that $u_1 \prec u_2$. Then we get
\[
p \leq p(u_1,u) \leq u(\dom u_1) = u_2(\dom u_1).
\]
Consequently, if we are given distinct $(u_1,i_1), (u_2,i_2) \in D$,
then by choosing $p$ large enough, we can ensure that there is no $V
\in \mathscr{G}_p$ for which $(u_1,i_1), (u_2,i_2) \in V$. This
establishes that the $\Lambda$-duplicate has a $\Gdelta$-diagonal.

We shall suppose for a contradiction that $D$ is Gruenhage. As
$\card D = \continuum$, we can use Proposition \ref{oddity} to find
a sequence $(U_n)_{n=1}^\infty$ of open subsets of $D$ so that
given any $t \in \Lambda$, there exists $n$ for which
\[
\{(t,1),(t,-1)\} \cap U_n
\]
is a singleton. Set
\[
E_{n,i} = \setcomp{t \in \Lambda}{(t,i) \in U_n \mbox{ and }(t,-i)
\notin U_n}.
\]
We know that $\Lambda = \bigcup_{n,i} E_{n,i}$. Now we are going to
decompose each $E_{n,i}$ into countably many subsets. If $t \in
E_{n,i}$ then $(t,i) \in U_n$, so there exists some $\theta(t) \prec
t$, $\theta(t) \in \{0,\varnothing\} \cup \Lambda^+$, such that
$W(\theta(t),t,i) \subseteq U_n$. Set
\[
E_{n,m,i} = \setcomp{t \in E_{n,i}}{\theta(t) \in A_m},
\]
where the $A_m$ are the antichains defined at the beginning of the section.
Suppose that $t,u \in E_{n,m,i}$ and $t \prec u$. Since
$\theta(t),\theta(u) \prec u$ are comparable and $\theta(t),\theta(u) \in A_m$,
it follows that $\theta(u) = \theta(t) \prec t \prec u$. Now, we have
\[
(t,j) \in W(\theta(u),u,i) \subseteq U_n
\]
where $j = (-1)^{\ell(t,u)}i$. Since $t \in E_{n,i}$, we gather that
$j=i$, whence $\ell(t,u)$ is an even number.

To simplify the notation, we shall alter the indices and denote the
$E_{n,m,i}$ by $E_n$, $n < \infty$. In summary, we have shown that
if $D$ is Gruenhage then we can write $\Lambda =
\bigcup_{n=1}^\infty E_n$, where each $E_n$ has the property that
$\ell(t,u)$ is an even number whenever $t,u \in E_n$ and $t \prec
u$. In the final part of the proof, we use a Baire category type
argument (cf.\ \cite[Lemma 10.1]{haydon:99}) to show that no
decomposition of $\Lambda$ into such sets $E_n$ is possible.

Set $\Lambda_1 = \Lambda$ and let $m_1$ be minimal, subject to the
condition that there exists some $t_1 \in \Lambda_1 \cap E_{m_1}$.
Let
\[
k_1 = \min \omega\setminus \ran t_1, \quad l_1 = \min
\omega\setminus (\ran t_1 \cup \{k_1\}), \quad u_1 = t_1 \cup
\{(\dom t_1,l_1)\}
\]
and define
\[
\Lambda_2 = \setcomp{v \in [u_1,\infty) \cap \Lambda_1}{k_1 \notin
\ran v}.
\]
We observe that $\Lambda_2 \cap E_{m_1}$ is empty. If $v \in
\Lambda_2$ then $v(\dom t_1) = u_1(\dom t_1) = l_1 \leq v(\eta)$ for
any $\eta \in [\dom t_1,\dom v)$, by minimality of $l_1$ and the
fact that $k_1 \notin \ran v$. Therefore, $\tau(t_1,v) = (\dom
t_1)$ and $\ell(t_1,v) = 1$. Since $t_1 \in E_{m_1}$ and
$\ell(t_1,v)$ is not an even number, we have $v \notin E_{m_1}$.

Continue by letting $m_2$ be minimal, subject to the condition that
we can find some $t_2 \in \Lambda_2 \cap E_{m_2}$. Necessarily $m_2
> m_1$. Let
\[
k_2 = \min\omega\setminus (\ran t_2 \cup \{k_1\}) > l_1, \quad l_2 =
\min\omega\setminus (\ran t_2 \cup \{k_1,k_2\}),
\]
$u_2 = t_2 \cup \{(\dom t_2, l_2)\}$ and define
\[
\Lambda_3 = \setcomp{v \in [u_2,\infty) \cap \Lambda_2}{k_2 \notin
\ran v}.
\]
As above, we find that $\Lambda_3 \cap E_{m_2}$ is empty because if
$v \in \Lambda_3$ then $\tau(t_2,v) = (\dom t_2)$ and $\ell(t_2,v)=1$,
however $t_2 \in E_{m_2}$ and $\ell(t_2,v)$ must be even if $v$ is to be
an element of $E_{m_2}$ as well.

Let $m_3 > m_2$ be minimal, subject to there being some $t_3 \in
\Lambda_3 \cap E_{m_3}$, and define
\[
k_3 = \min\omega\setminus (\ran t_3 \cup \{k_1,k_2\}) > l_2, \quad
l_3 = \min\omega\setminus (\ran t_3 \cup \{k_1,k_2,k_3\}).
\]
It should be clear how to proceed. We obtain a decreasing sequence
of sets $(\Lambda_j)_{j=1}^\infty$ and corresponding least elements
$u_j$, with the property that $\Lambda_{j+1} \cap E_m$ is empty
whenever $m \leq m_j$. Moreover, if $v \in \Lambda_{j+1}$ then
$k_1,\dots k_j \notin \ran v$.

Let $u = \bigcup_{j=1}^\infty u_j$. Being the union of an increasing
sequence of injective functions, $u$ is also injective. By
construction, we have ensured that $k_j \notin \ran u$ for all $j$,
whence $\omega\setminus\ran u$ is infinite and $u \in \Lambda$.
Moreover, $u \in \Lambda_j$ for all $j$. However, this means that
$u \notin E_m$ for any $m$, because the $m_j$ form a strictly
increasing sequence. This contradiction establishes the fact that
$D$ is not Gruenhage.
\end{proof}

\begin{cor}
The 1-point compactification $K$ of $D$ is a scattered compact
non-Gruenhage space with {\st}. By Theorem \ref{scattered} (but not
Theorem \ref{smithgru}), $C(K)$ admits an equivalent lattice norm
with a strictly convex dual norm.
\end{cor}

\section{The space $C_0(D)$ has a $C^\infty$-smooth bump}

If $L$ is locally compact and scattered then the Banach space
$C_0(L)$ of continuous real-valued functions vanishing at infinity
is an {\em Asplund} space. Recently, a consistent negative solution
was given to the long-standing problem of whether every Asplund
space admits a {\em $C^1$-smooth bump function}, that is, a non-zero
real-valued continuously Fr\'echet differentiable function
which vanishes outside some norm-bounded set \cite{jlotod:11}. As
far as the author is aware, the question of whether such an Asplund
space can be found in ZFC, or whether an example of type $C_0(L)$
can exist, remains open. Thus, it makes sense to test $C_0(L)$
whenever a new locally compact scattered space $L$ comes along. The
purpose of this final section is to confirm that (unfortunately!)
$C_0(D)$ does admit such a function.

\begin{defn}\label{to}
Given a non-empty set $\Gamma$, we say that
$\mapping{T}{C_0(L)}{\czerok{L \times \Gamma}}$ is a (generally non-linear)
{\em Talagrand operator of class $C^\infty$} if
\begin{enumerate}
\item whenever $f \in C_0(L)$ is non-zero then there exists $(t,\gamma) \in L \times \Gamma$
such that $|f(t)| = \pnorm{f}{\infty}$ and $(Tf)(t,\gamma) \neq 0$,
and
\item for every pair $(t,\gamma)$, the map $f \mapsto (Tf)(t,\gamma)$
is $C^\infty$-smooth, i.e., has Fr\'echet derivatives of all orders,
on the set on which it is non-zero.
\end{enumerate}
\end{defn}

It follows from \cite[Corollary 3]{haydon:96} that if $C_0(L)$
admits such an operator then it admits a $C^\infty$-smooth bump
function. We shall prove that $C_0(D)$
admits such an operator. Our method follows that of \cite[Theorem
9.3]{haydon:99}, which shows that $C_0(\Upsilon)$ admits a $C^\infty$-smooth
bump function for every tree $\Upsilon$. However, since the
topology of $D$ is slightly more complicated than that of ordinary trees,
we present some of the details.

\begin{lem}\label{refinement}
Suppose that $U$ and $V$ are open subsets of $D$ such that the
restrictions $\pi\restrict{U}$ and $\pi\restrict{V}$ of the natural
projection $\pi$ are injective, and $\pi(U) = \pi(V) = (r,t]$ for some
$r \in \Lambda \cup \{0\}$, $t \in \Lambda$.
Then there exist basic open sets $W_1,\dots, W_k$ and
$W^\prime_1,\dots, W^\prime_k$ such that
\[
U = W_1 \cup \dots \cup W_k \quad\text{and}\quad V = W^\prime_1 \cup
\dots \cup W^\prime_k,
\]
and given any $i \leq k$, either $W_i = W^\prime_i$ or $W_i \cap
W^\prime_i$ is empty. Moreover, if $i \neq j$ then both $W_i \cap
W_j$ and $W^\prime_i \cap W^\prime_j$ are empty.
\end{lem}

\begin{proof}
The argument is similar to the one used to show that the basis elements
are compact.
Set $t_1 = t$ and take $p_1,q_1 \in \{1,-1\}$ such that $(t_1,p_1)
\in U$ and $(t_1,q_1) \in V$. From Proposition
\ref{basis}, we can find $t_2 \in [r,t_1)$ such that $W(t_2,t_1,p_1)
\subseteq U$ and $W(t_2,t_1,q_1) \subseteq V$. Set $W_1 =
W(t_2,t_1,p_1)$ and $W^\prime_1 = W(t_2,t_1,q_1)$. If $p_1=q_1$ then $W_1 =
W^\prime_1$, and if not then $W_1 \cap W^\prime_1$ is empty. If $t_2
= r$ then stop. Otherwise, continue by finding $p_2,q_2 \in \{1,-1\}$
and $t_3 \in [r,t_2)$ such that $W(t_3,t_2,p_2)
\subseteq U$ and $W(t_3,t_2,q_2) \subseteq V$. This process stops at a finite stage $k$. Since $\pi$ is
injective on $U$, we have $U = W_1 \cup \dots \cup W_k$, and
similarly for $V$.
\end{proof}

Given $(s,i) \in D$, we define the set of `immediate successors'
$(s,i)^+ = s^+ \times \{1,-1\}$. In the next lemma, we gather together some
properties of elements in $C_0(D)$ that we need in order to define
our Talagrand operator.

\begin{lem}\label{prop}
Let $f \in C_0(D)$ and $\delta > 0$.
\begin{enumerate}
\item Given $(s,i) \in D$, there are only finitely many $(t,j) \in
(s,i)^+$ satisfying $|f(t,j)| \geq \delta$.
\item If $f$ is non-zero then there exists maximal $s \in \Lambda$,
subject to there being $i \in \{1,-1\}$ satisfying $|f(s,i)| = \pnorm{f}{\infty}$.
\item For all but finitely many $(s,i) \in D$, there exists
$(t,i) \in (s,j)^+$ such that $|f(s,i) - f(t,j)| < \delta$.
\end{enumerate}
\end{lem}

\begin{proof}$\;$
\begin{enumerate}
\item Observe that
\[
K = \setcomp{(t,j) \in D}{|f(t,j)| \geq \delta} \cap (s,i)^+
\]
is compact and discrete, hence finite.
\item  If $f$ is non-zero then
\[
M = \setcomp{(s,i) \in D}{|f(s,i)| = \pnorm{f}{\infty}}
\]
is compact, and thus there exist finitely many pairs $(s_k,i_k)
\in M$, $k \leq n$, such that $M \subseteq \bigcup_{k=1}^n W(0,s_k,i_k)$.
Take maximal $s$ amongst the $s_k$.
\item The intersection of any two basis elements of $D$ is a finite union
of pairwise disjoint basis elements. Hence, by a standard
Stone-Weierstrass argument, $C_0(D)$ is equal to the closed linear
span of the family of indicator functions $\ind{W}$, as $W$ ranges
over the basis elements.

Thus we are done by uniform approximation if we can show that (3)
applies to finite linear combinations of the $\ind{W}$. Let
$W_1,\dots, W_n$ be basis elements, $a_1,\dots, a_n \in \real$ and
set $f = \sum_{k=1}^n a_k \ind{W_k}$. By splitting the $W_k$ into
smaller basis elements if necessary, and by using Lemma
\ref{refinement}, we can assume that whenever $k \neq l$, either
$W_k = W_l$ or $W_k \cap W_l$ is empty.

If $\pi(W_k) = (r_k,t_k]$, $k \leq n$, then set $F =
\{\oneton{r}{n}\} \cup \{\oneton{t}{n}\}$. Take any $(s,i) \in D$
satisfying $s \notin F$. Define $E$ to be the set of $k \leq n$ such
that $(s,i) \in W_k$, so that $f(s,i) = \sum_{k \in E} a_i$. From
above, we know $W_k = W_l$ and $t_k = t_l$ whenever $k,l \in E$. If
$E$ is non-empty then let's denote this common set and endpoint by
$W$ and $u$, respectively. Because $(s,i) \in W$ and $s \notin F$,
we have $s \prec u$. Take $t \in s^+$ such that $t \preccurlyeq u$,
and then $j \in \{1,-1\}$ such that $(t,j) \in W$. It is clear that
$(t,j) \notin W_k$ whenever $k \notin E$, else $W_k = W \ni (s,i)$.
In conclusion, $f(t,j) = \sum_{k \in E} a_k = f(s,i)$. If $E$ is
empty then pick any $t \in s^+$. If $(t,1) \notin W_k$ for all $k$
then we are done because $f(t,1) = 0 = f(s,i)$. If $(t,1) \in W_k$
for some $k$ then we claim that $(t,-1) \notin W_l$ for any $l$.
Certainly, $(t,-1) \notin W_k$. We have $r_k \prec t \preccurlyeq
t_k$, meaning $r_k \preccurlyeq s$, and as $s \notin F$ we know that
$r_k \prec s$. Because $(s,i) \notin W_k$, we must have $(s,-i) \in
W_k$ instead. Suppose that $(t,-1) \in W_l$ for some $l$. Then by
the same argument we have $r_l \prec s \prec t_l$ and $(s,-i) \in
W_l$. However, this implies $(t,-1) \in W_l = W_k$, which isn't so.
Therefore $(t,-1) \notin W_l$ for all $l$ and $f(t,-1) = 0 =
f(s,i)$.
\end{enumerate}
\end{proof}

\begin{prop}
The space $C_0(D)$ admits a Talagrand operator of type $C^\infty$.
\end{prop}

\begin{proof}[Sketch proof]
We define $\mapping{T}{C_0(D)}{\czero(D \times \nat)}$ in almost
exactly the same way as in \cite[Theorem 9.3]{haydon:99}. Let
$\mapping{\phi}{\real}{[0,1]}$ be an even $C^\infty$-smooth function
satisfying $\phi(x) = 0$ for $|x|\leq \frac{1}{2}$ and $\phi(x) = 1$
for $|x| \geq 1$. Set $\psi = 1- \phi$. Given $f \in C_0(D)$, $(s,i) \in D$ and
$n \in \nat$, define
\[
(Tf)(s,i,n) = \left\{
\begin{array}{l} \text{$0$ if $f(s,i)=0$ or if there is $(t,j) \in (s,i)^+$ with $f(t,j) = f(s,i)$,}\\
{\displaystyle 2^{-n}\phi(2^nf(s,i))\prod_{(t,j) \in (s,i)^+} \psi\left(\frac{2^{-n}f(t,j)}{f(t,j)-f(s,i)}\right)} \text{ otherwise.}
\end{array}\right.
\]
To verify that $T$ is indeed a Talagrand operator of class $C^\infty$,
we simply use Lemma \ref{prop} and follow the proof of
\cite[Theorem 9.3]{haydon:99}, replacing $s$ by $(s,i)$ and $t$ by $(t,j)$
throughout.
\end{proof}

By \cite{hh:07}, it follows that $C_0(D)$ also admits $C^\infty$-smooth
partitions of unity.

\bibliographystyle{amsplain}

\begin{thebibliography}{10}

\bibitem{gru:84} G.\ Gruenhage, {\em Generalized metric spaces},
in: Handbook of Set-Theoretic Topology, K.\ Kunen and J.\ Vaughan
(eds.), North-Holland, 1984, 235--293.

\bibitem{gru:87} G.\ Gruenhage, {\em A note on Gul'ko compact
spaces}, Proc.\ Amer.\ Math.\ Soc. \textbf{100} (1987), 371--376.

\bibitem{hh:07} P.\ H\'{a}jek and R. Haydon, {\em Smooth norms and
approximation in Banach spaces of the type $C(K)$}, Q.\ J.\ Math
\textbf{58} (2007), 221--228.

\bibitem{hart:82} K.\ P.\ Hart, {\em Characterizations of
$\real$-embeddable and developable $\wone$-trees}, Nederl.\ Akad.\
Wetensch.\ Indag.\ Math. \textbf{44} (1982), 277--283.

\bibitem{haydon:96} R.\ Haydon, {\em Smooth functions and partitions of unity on
certain Banach spaces.} Q.\ J.\ Math. \textbf{47} (1996), 455--468.

\bibitem{haydon:99} R.\ G.\ Haydon, {\em Trees in renorming theory}, Proc.
London Math.\ Soc. \textbf{78} (1999), 541--584.

\bibitem{jlotod:11} J.\ Lopez-Abad and S.\ Todor\v cevi\'c,
{\em Generic Banach spaces and generic simplexes}, preprint.

\bibitem{ost:11} J.\ Orihuela, R.\ J.\ Smith and S.\ Troyanski,
{\em Strictly convex norms and topology}, preprint.\\
{\tt http://arxiv.org/abs/1012.5595v1}

\bibitem{smith:09} R.\ J.\ Smith, {\em Gruenhage compacta and
strictly convex dual norms}, J.\ Math.\ Anal.\ Appl. \textbf{350}
(2009), 745--757.

\bibitem{st:10} R.\ J.\ Smith and S.\ Troyanski, {\em Renormings of
C(K) spaces}, Rev.\ R.\ Acad.\ Cienc.\ Exactas F\'is.\ Nat.\ Ser.\ A
Mat.\ \textbf{104} (2010), 375--412.

\end{thebibliography}

\end{document}